\documentclass[11pt]{article}
\usepackage[T1]{fontenc}
\usepackage[utf8]{inputenc}
\usepackage{lmodern} % "Reverts" font changes caused by the above command
\usepackage{amsmath,amsfonts,amssymb,amsthm}
\usepackage{graphicx}
\usepackage{float}
\usepackage{natbib}

\usepackage[a4paper,top=2.5cm,bottom=2.5cm,left=2.5cm,right=2.5cm]{geometry}
\usepackage[colorlinks=true, allcolors=blue]{hyperref}
\usepackage{caption}
\usepackage[justification=centering]{subcaption}

\usepackage{booktabs}
\usepackage[table]{xcolor}

\usepackage{xspace}

\usepackage{tikz}

\usepackage{listings}
\usepackage{threeparttable}
\usepackage{multirow}
\usepackage{siunitx}

\usepackage{authblk}

\definecolor{darkviolet}{rgb}{0.58, 0.0, 0.83}

\newcommand{\ndelta}[2]{{#1}^{#2}}
\newcommand{\x}{\boldsymbol{x}}

\usepackage[normalem]{ulem}

\providecommand{\card}[1]{\lvert#1\rvert}

\newcommand{\solver}[1]{\texttt{#1}}

\usepackage{mathtools}

\usepackage[skins]{tcolorbox}

\usepackage{enumitem}
\setlist[itemize,3]{label=$\diamond$}

\newenvironment{shortitem}[1][]{\begin{itemize}[topsep=0pt, itemsep=0pt, parsep=0.5pt, leftmargin=*, #1]}{\end{itemize}}

\newcommand{\R}{\mathbb R}
\newcommand{\N}{\mathbb N}

\newcommand{\itercounter}{\tau}

\newcommand{\cola}{Q}

\newcommand{\X}{\boldsymbol{X}}
\newcommand{\xb}{\bar{\boldsymbol{x}}}
\newcommand{\Xb}{\bar{\boldsymbol{X}}}

\newcommand{\zbar}{\bar z}

\renewcommand{\stop}{\textsc{End}}
\newtheorem*{lem}{Lemma~2}
\newtheorem{lemma}{Lemma}

\theoremstyle{definition}
\newtheorem{examp}{Example}
\newenvironment{example}{\begin{examp} \myendexampletrue}{\ifmyendexample\myendexample\fi \end{examp}}

\let\myendexample=\relax
\def\myexamplesymbol#1{\def\myendexample
 {{\unskip\nobreak\hfil\penalty50
   \quad\hbox{}\nobreak\hfil #1
   \parfillskip=0pt \finalhyphendemerits=0 \par}}}
\myexamplesymbol{$\Diamond$}
\newif\ifmyendexample

\usepackage{todonotes}

\title{A note on the convergence guarantees of RLT-based algorithms for polynomial optimization} %{\small \julio{ALTERNATIVE: RLT-based algorithms for polynomial optimization: a theoretical (and practical) nuance}}}
\author[1,2]{Alejandro Barros-González}
\author[2,3]{Julio González-Díaz\thanks{Corresponding author: julio.gonzalez@usc.es}}
\author[2,4]{Brais González-Rodríguez}
\author[2,3]{Iria Rodríguez-Acevedo}
\affil[1]{University of Vigo}
\affil[2]{CITMAga (Galician Center for Mathematical Research and Technology), 15782 Santiago de Compostela, Spain}
\affil[3]{Department of Statistics, Mathematical Analysis and Optimization and MODESTYA Research Group, University of Santiago de Compostela, 15782 Santiago de Compostela, Spain}
\affil[4]{Department of Statistics and Operational Research and SiDOR Research Group, University of Vigo}

\date{\today}

\begin{document}

\maketitle

\begin{abstract}
This paper identifies and addresses a mathematical oversight in one of the foundational results on the Reformulation-Linearization Technique (RLT) for polynomial optimization. Correctness of the original result can be easily recovered by adding a minor and natural assumption. However, not being aware of this nuance may lead to the loss of convergence guarantees in RLT-based algorithms, especially if they rely on filtering strategies for the construction of the bound-factor constraints. 
\end{abstract}


\medskip


\section{Introduction}

The Reformulation-Linearization Technique (RLT), as introduced by \cite{Sherali1992}, serves as the mathematical foundation for numerous global optimization frameworks targeting continuous polynomial programming. Notable software implementations of this technique include the global solvers \solver{RLT-POS} \citep{Dalkiran2016} and \solver{RAPOSa} \citep{Gonzalez-Rodriguez:2023,raposaconic}.

The so-called bound-factor constraints constitute a core element of the RLT framework. They must be added to the formulation to tighten the ensuing linear relaxations and guarantee the convergence of the spatial branch-and-bound algorithm as the variable domains are sequentially partitioned. However, a major computational bottleneck of RLT-based schemes stems from the fact that the number of these constraints grows exponentially with the size of the problem. Consequently, numerous filtering and reduction strategies have been proposed in the literature to mitigate this combinatorial explosion \citep{sherali1997,Sherali2013,dalkiran2018,gonzalezrodriguez2025degree}.

This combinatorial burden was already anticipated in the seminal work of \cite{Sherali1992}, who presented a result, Lemma~2, stating that, given a polynomial optimization problem of degree $\delta$ (defined by the highest degree of any monomial appearing in the formulation), the bound-factor constraints of degree strictly smaller than $\delta$ are redundant, as they are implied by those of degree exactly $\delta$. In this paper, we reveal an inaccuracy in the proof of the latter result and introduce a minor, yet necessary, missing assumption to restore its validity. Importantly, Lemma~2 in \cite{Sherali1992} is also a core element of the convergence results behind the filtering strategies in \cite{Sherali2013}, which are themselves at the core of solvers like \solver{RLT-POS} and \solver{RAPOSa}. We discuss how the inaccuracy in Lemma~2 compromises the convergence results in \cite{Sherali2013} and, therefore, ignoring this issue can lead to a loss of global convergence guarantees in RLT-based algorithms, and especially so if they build upon filtering strategies. Finally, we briefly discuss why this theoretical oversight has remained unnoticed in the literature without impacting the practical behavior of RLT-based algorithms.

\section{The Reformulation-Linearization Technique}\label{sec:rlt}
We consider (continuous) polynomial optimization problems given by
\begin{equation}
\begin{aligned}
\text{minimize} & \quad \phi_0(\x)\\
\text{subject to}  & \quad \phi_r(\x)\ge \beta_r, & r=1,2,\ldots, m_1 \\
& \quad \phi_r(\x)=\beta_r, & r=m_1+1,\ldots,m\\
& \quad \x\in\Omega \subsetneq \mathbb{R}^{\card{N}}\text{,}
\end{aligned}
\label{eq:PO}
\tag*{$PP(\Omega)$}
\end{equation}
where $N$ denotes the set of variables, each $\phi_r(\mathbf{x})$ is a polynomial of degree $\delta_r \in \mathbb{N}$, and the set $\Omega = \{ \x \in \mathbb{R}^{\card{N}}: 0 \leq l_j \leq x_j \leq u_j < \infty, \, \forall j \in N \} \subsetneq \mathbb{R}^{\card{N}}$ is a hyperrectangle containing the feasible region. The degree of the problem is given by $\delta=\max_{r \in \{0,\ldots,m\}} \delta_r$.

%A monomial can be represented as a multiset of variable indices $J$ on $N$, where index multiplicity dictates the exponents and multiset addition corresponds to monomial multiplication.

A multiset $J$ over $N = \{1, \dots, n\}$ is defined by its multiplicity function $\mu_J: N \to \mathbb{N}_0$. For multisets $J_1, J_2$ with $J_2 \subseteq J_1$, their sum $J_1 + J_2$ and difference $J_1 - J_2$ are defined via:
\[
\mu_{J_1 + J_2}(i) = \mu_{J_1}(i) + \mu_{J_2}(i) \quad \text{and} \quad\mu_{J_1 - J_2}(i) = \mu_{J_1}(i) - \mu_{J_2}(i), \quad \forall i \in N.
\]
There is natural correspondence between multisets and monomials, with the above operations reflecting monomial multiplication and division:
\[
\textstyle
\prod_{j \in J_1+ J_2} x_j= \prod_{j \in J_1}x_j \cdot \prod_{j \in J_2}x_j \quad \text{and} \quad \prod_{j \in J_1 - J_2} x_j = \prod_{j \in J_1}x_j / \prod_{j \in J_2}x_j.
\]
Let $\ndelta{N}{\lambda}$ denote the collection of all multisets with cardinality $\lambda$ (monomial degree).

The RLT linearization process substitutes each monomial product with an independent continuous variable:
\begin{equation}
X_J = \prod_{j \in J}x_j.
\label{eq:identity}
\end{equation}
Crucial to the relaxation are the bound-factor constraints, hereafter BF-constraints. Given $\lambda\in \mathbb N$ and multisets $J_1, J_2$ such that $J_1 + J_2 \in \ndelta{N}{\lambda}$, the corresponding constraint of degree $\lambda$ is defined as:
\begin{equation*}
F_{\Omega}^{\lambda}(J_1,J_2)=\prod_{j\in J_1}{(x_j-l_j)}\prod_{j\in J_2}{(u_j-x_j)}\ge 0.
%\label{eq:BFC}
\end{equation*}
Denoting $[\phi(\x)]_L$ as the linear operator replacing every nonlinear monomial in $\phi(\x)$ with its corresponding RLT variable $X_J$, the linear RLT relaxation of~\ref{eq:PO} is given by:
\begin{equation}
\begin{aligned}
\text{minimize} & \quad [\phi_0(\x)]_L & \\
\text{subject to}  & \quad [\phi_r(\x)]_L \ge \beta_r, & r=1,2,\ldots, m_1 \\
& \quad [\phi_r(\x)]_L =\beta_r, & r=m_1+1,\ldots,m\\
& \quad [F_{\Omega}^{\delta}(J_1,J_2)]_L \ge 0, & J_1 + J_2 \in \ndelta{N}{\delta}\\
& \quad \x\in\Omega \subsetneq \mathbb{R}^{\card{N}}\text{.}\\
\end{aligned}
\label{eq:LP}
\tag*{$LP(\Omega)$}
\end{equation}

In Figure~\ref{fig:RLTalg} we present the general scheme of an RLT-based algorithm, which proceeds by defining a series of linear relaxations, $LP(\Omega^k)$, of polynomial optimization problems analogous to~\ref{eq:PO}, but with respect to different hyperrectangles $\Omega^k$. For each $j \in N$ and each $k \in \N$, let $l^k_j$ and $u^k_j$ denote the lower and upper bounds of variable $x_j$ in $LP(\Omega^k)$. Importantly, since BF-constraints of $LP(\Omega^k)$ depend on the $l^k_j$ and $u^k_j$ bounds, they vary across nodes. Once a linear relaxation $LP(\Omega^k)$ has been solved and an optimal solution $(\xb, \Xb)$ is available, $\theta_j^k$ denotes a measure of the violation of the RLT-defining identities in Eq.~\eqref{eq:identity} in which variable $j \in N$ is involved.\footnote{Refer, for instance, to Section~5.4 in \cite{Gonzalez-Rodriguez:2023} and to \cite{ghaddar2023} for additional details.} The convergence of this branch-and-bound algorithm follows from Theorem~1 in \cite{Sherali1992}.

\begin{figure}[!htbp]
\centering
\begin{tcolorbox}
\small
\textbf{Initialization.} Let $LB\vcentcolon=-\infty$, $UB\vcentcolon=+\infty$, and $\x^{best}\vcentcolon=\emptyset$. Let $\itercounter\vcentcolon=1$, $\Omega^1 \vcentcolon= \Omega$, $\cola\vcentcolon=\{1\}$, and $LB^1=-\infty$.

\textbf{Stage 1 \emph{(main)}.} Choose $k\in \cola$ such that $LB^k=\min_{s\in \cola} LB^s$. Let $\cola \vcentcolon= \cola \backslash \{ k\}$. Solve problem $LP(\Omega^k)$.
	\begin{shortitem}
		\item If $LP(\Omega^k)$ is infeasible, go to {\bf Stage~2}.
		\item If $LP(\Omega^k)$ is feasible, let $\zbar^k$ be the optimal value and let $(\xb^k, \Xb^k)$ be an optimal solution.
	\begin{shortitem}
		\item If $\zbar^k<UB$ and $\theta_j^k = 0$ for all $j \in N$:
		\begin{shortitem}
			\item \emph{Update upper bound.} $UB\vcentcolon=\zbar^k$. Let $\x^{best}\vcentcolon=\xb^k$.
			\item \emph{Prune.} Remove all $s\in \cola$ such that $LB^s \ge UB$. Go to {\bf Stage~2}.
		\end{shortitem} 
		\item If $\zbar^k < UB$ and $\theta_j^k > 0$ for some $j \in N$:
		\begin{shortitem}
			\item \emph{Branch.} Choose $p \in N$ such that $\theta_p^k = \max_{j \in N} \theta_j^k$. Branch at $\beta \vcentcolon= \bar x^k_p$, by defining $\Omega^{\itercounter+1} \vcentcolon= \Omega^k \cap \{ \x \in \R^{\card{N}} \text{ s.t. } l^k_p \leq x_p \leq \beta \}$ and $\Omega^{\itercounter+2} \vcentcolon= \Omega^k \cap \{ \x \in \R^{\card{N}} \text{ s.t. } \beta \leq x_p \leq u^k_p \}$.
	%and let $LP(\Omega^{\itercounter+1})$ and $LP(\Omega^{\itercounter+2})$ be the corresponding problems.
	\item Update queue $\cola \vcentcolon= \cola \cup \{ \itercounter+1, \itercounter+2\}$. Let $LB^{\itercounter+1} = LB^{\itercounter+2} \vcentcolon= \zbar^k$. Let $\itercounter \vcentcolon= \itercounter + 2$. Go to {\bf Stage~2}.
		\end{shortitem}
	\item If $\zbar^k \ge  UB$. \emph{Prune branch.} Go to {\bf Stage~2}.
	\end{shortitem}
\end{shortitem}

\textbf{Stage 2 \emph{(control)}.}
\emph{Update lower bound.} $LB\vcentcolon=\min\{\min_{k\in \cola}{LB^k}, UB \}$.
\begin{shortitem}
	\item If $LB=UB$, \stop:
	\begin{shortitem}
		\item If $\x^{best}=\emptyset$, \ref{eq:PO} is infeasible.
		\item If $\x^{best}\neq\emptyset$, $\x^{best}$ is a global optimum and $UB$ is the optimal value of \ref{eq:PO}.
	\end{shortitem} 
	\item Otherwise, go to {\bf Stage~1}.
\end{shortitem}
\end{tcolorbox}
\caption{Basic scheme of an RLT-based algorithm.}
\label{fig:RLTalg}
\end{figure}

\section{Filtering strategies}\label{sec:filtering}

To alleviate the exponential proliferation of BF-constraints in RLT-based schemes, \cite{Sherali2013} introduced two families of filtered BF-constraints, called $J$-set strategy and $N_J^{d_J}$-set strategy, alongside convergence results. Owing to its demonstrated efficiency, the $J$-set approach has become a standard filtering strategy in modern RLT implementations.

\subsection{$J$-set filtering strategy}
Let $\mathcal{M}$ denote the set of multisets corresponding to all nonlinear monomials present in \ref{eq:PO}. A first strategy discussed in \cite{Sherali2013} consists of generating just the BF-constraints associated with the elements of $\mathcal{M}$. Then, building upon Lemma~2 in \cite{Sherali1992}, they propose the $J$-set routine, which avoids the generation of dominated BF-constraints. To this end, they define $\mathcal{J}$ as the set of maximal elements in $\mathcal{M}$ under multiset inclusion. Under the $J$-set routine, BF-constraints are generated exclusively for the elements of $\mathcal{J}$. Formally, for each $J\in \mathcal{J}$, the following BF-constraints are generated:
  \[
	[\prod_{j \in J_1} (x_j - l_j) \prod_{j \in J_2} (u_j - x_j)]_L \ge 0, \quad \forall (J_1 + J_2) = J.
	\]

\subsection{$N_J^{d_J}$-set filtering strategy}

For a given multiset $J$ with cardinality $d_J = \card{J}$, let $N_J \subseteq N$ be the set of distinct variable indices appearing in $J$. Then, $N_J^{d_J}$ represents the collection of all multisets with cardinality $d_J$ on the variables in $N_J$. Note that $J\in N_J^{d_J}$, but $N_J^{d_J}$ contains more multisets. For instance, given monomial $x_1^2 x_2$, then $N_J=\{1,2\}$ and $N_J^{d_J}$ also contains the multisets associated with monomials $x_1^3$, $x_2^3$, and $x_1 x_2^2$. The idea of the $N_J^{d_J}$-set approach is to generate BF-constraints exclusively for the elements of $N_J^{d_J}$. This would correspond to generating, for each $ J' \in N_J^{d_J}$, the following BF-constraints:
 \[[\prod_{j \in J_1} (x_j - l_j) \prod_{j \in J_2} (u_j - x_j)]_L \ge 0, \quad \forall (J_1 + J_2) =J'.\]  
However, since this approach might generate many duplicated and dominated BF-constraints, \cite{Sherali2013} proposed the following $N_J^{d_J}$-set routine:
\begin{description}
    \item[Initialization:] Construct the collection of distinct subsets $\mathcal{K} = \{N_{J} : N_{J}$ is nonrepetitive with $N_{J} \subseteq \bar{J} \text{ for some } \bar{J} \in \mathcal{M}\}$.
    \item[Step 1:] For each $N_{J} \in \mathcal{K}$, compute the value $g(N_{J}) = \max\{\card{\bar{J}} : \bar{J} \in \mathcal{M} \text{ and } N_{J} \subseteq \bar{J}\}$.
    \item[Step 2:] For each $N_{J} \in \mathcal{K}$, generate the following BF-constraints:
    \[
    [ \prod_{j \in J_1} (x_j - l_j) \prod_{j \in J_2} (u_j - x_j)]_L \ge 0, \quad \forall (J_1 + J_2) \in N_{J}^{g(N_{J})} \text{ with } N_J \subseteq J_1 + J_2.
    \]
\end{description}
Note that, in Step 2 above, only the elements of $N_{J}^{g(N_{J})}$ with $N_J \subseteq J_1 + J_2$ are considered. Importantly, when $J \in \mathcal{M}$ and $J + \{t\} \in \mathcal{M}$ for some variable $t \in N$, this routine does not construct BF-constraints containing $N_J$ at degree $d_J$ or lower. By skipping those constraints of degree $d_J$ associated with $N_J^{d_J}$, the routine is implicitly relying on Lemma~2 to prevent the generation of dominated constraints much like the standard RLT framework, which restricts bound-factor generation strictly to degree $\delta$.

\section{The theoretical oversight}

Below, we restate Lemma~2 in \cite{Sherali1992} and provide a faithful transcription of its proof, adapted to our notation.

\begin{lem}[\cite{Sherali1992}] Consider $\delta'\in \mathbb N$, with $1\le\delta'<\delta$. Then, constraints $[F_{\Omega}^{\delta'}(J_1,J_2)]_L\ge 0$, where $J_1+J_2\in N^{\delta'}$ are all implied by those of degree $\delta$, $[F_{\Omega}^{\delta}(J_1,J_2)]_L\ge 0$, where $J_1+J_2\in N^{\delta}$.
\end{lem}

\begin{proof}[Proof (Transcription of the proof in \cite{Sherali1992}).]
 For any $\delta'$, $1 \leq \delta' < \delta $, consider the surrogate of the constraints $[F_{\Omega}^{\delta'+1}(J_1+ \{t\},J_2)]_L\ge 0$ and $[F_{\Omega}^{\delta'+1}(J_1,J_2+ \{t\})]_L\ge 0$, where $J_1+J_2\in N^{\delta'}$ and $t\in N$.
{\footnotesize
\[
\begin{array}{l}
[F_{\Omega}^{\delta'+1}(J_1+ \{t\},J_2)]_L + [F_{\Omega}^{\delta'+1}(J_1,J_2+ \{t\})]_L= [(x_t-l_t)F_{\Omega}^{\delta'}(J_1,J_2)]_L+[(u_t-x_t)F_{\Omega}^{\delta'}(J_1,J_2)]_L  \\
\qquad = [x_t F_{\Omega}^{\delta'}(J_1,J_2)]_L - [l_t F_{\Omega}^{\delta'}(J_1,J_2)]_L + [u_t F_{\Omega}^{\delta'}(J_1,J_2)]_L -[x_t F_{\Omega}^{\delta'}(J_1,J_2)]_L = [(u_t-l_t) F_{\Omega}^{\delta'}(J_1,J_2)]_L\ge 0.
\end{array}
\]
}
Since $(u_t-l_t)\ge 0$, it follows that $[F_{\Omega}^{\delta'}(J_1,J_2)]_L\ge 0$ is implied by $[F_{\Omega}^{\delta'+1}(J_1+ \{t\},J_2)]_L\ge 0$ and $[F_{\Omega}^{\delta'+1}(J_1,J_2+ \{t\})]_L\ge 0$. The required result follows by the principle of induction, and this completes the proof.
\end{proof}

The oversight in the proof occurs in the first statement of the last paragraph. For the assertion to hold, $(u_t-l_t)\ge 0$ must be replaced with $(u_t-l_t)>0$. Importantly, as we illustrate in Example~\ref{example_orig} below, it is not only that the argument in the proof is incorrect: the statement in Lemma~2 is not valid unless it incorporates the assumption that ``there is $t\in N$ with $(u_t-l_t)>0$''. One might be tempted to argue that this assumption could simply be incorporated when defining the hyperrectangle $\Omega$. However, this last option would not solve the issue since, to guarantee the convergence of the RLT-based algorithm in Figure~\ref{fig:RLTalg} to a global optimum, Lemma~2 must hold not only with respect to $LP(\Omega)$ at the root node, but also with respect to all $LP(\Omega^k)$ relaxations. Below we present an extreme situation where this might be an issue.

\begin{example}\label{example_orig}
Suppose that, at some node in an RLT-based algorithm, we must solve the RLT relaxation of the following polynomial optimization problem:
\[
\begin{array}{ll}
\min & x_1^3\\[0.1mm]
\text{s.t.}
& 1\le x_1\le 1.
\end{array}
\]
Introducing the RLT variables $X_{11}$ and $X_{111}$, and including the linearized BF-constraints associated with \(x_1x_1x_1\) yields, after some straightforward computations, the following linear relaxation:
\[
\begin{array}{ll}
\min & X_{111} \\[0.1mm]
\text{s.t.}
& X_{111} - 3X_{11} + 3x_{1} = 1 \\
& 1\le x_1\le 1.
\end{array}
\]
This problem is unbounded since, given $\alpha \in \mathbb R$, all solutions of the form $(x_1,X_{11},X_{111})=(1,(2+\alpha)/3,\alpha)$ are feasible. However, the linearized BF-constraints for the degree-2 monomial $x_1x_1$ reduce to:
\[
X_{11}-2x_1+1=0.
\]
Since $1\le x_1\le 1$, this additional constraint forces $X_{11}=1$ and the feasible region collapses to $(x_1,X_{11},X_{111})=(1,1,1)$, yielding an optimal value of $1$. Hence, the linearized BF-constraints for $x_1x_1x_1$ fail to imply those for $x_1x_1$.
\end{example}

Importantly, if an RLT-based algorithm is applied using a filtering strategy, then additional care is needed since, given a multiset $J$ with $\card{J}=\delta-1$ and a variable $t\in N$, it is no longer guaranteed that the BF-constraints associated with monomial $J+\{t\}$ are present. Consequently, as shown in the example below, restoring the validity of Lemma~2 in these settings requires additional conditions, since situations far less extreme than Example~\ref{example_orig} can can lead to similarly undesirable behavior.

\begin{example}\label{example_Jset}
Suppose that, at some node in an RLT-based algorithm with the $J$-set filtering strategy, we must solve the RLT relaxation of the following polynomial optimization problem:
\[
\begin{array}{ll}
\min & x_1x_2x_3-x_2x_3\\[0.1mm]
\text{s.t.}
& x_2 \leq x_1 \\
& 0\le x_1\le 0, \quad 0\le x_2\le 1, \quad 0\le x_3\le 1.
\end{array}
\]
In this simple example, the unique maximal multiset in $\mathcal{M}$ corresponds with the degree-3 monomial \(x_1x_2x_3\). Introducing the RLT variables $X_{12}$, $X_{13}$, $X_{23}$, $X_{123}$, and including the eight linearized BF-constraints associated with \(x_1x_2x_3\) yields, after some straightforward computations, the following linear relaxation:

\[
\begin{array}{ll}
\min & X_{123}-X_{23} \\[0.1mm]
\text{s.t.}
& x_2 \leq x_1 \\[0.1mm]
& X_{123} - X_{12} - X_{13} + x_1 = 0 \\
& X_{123} -X_{12} = 0 \\
& X_{123} -X_{13} = 0 \\
& X_{123} = 0 \\
& 0\le x_1\le 0, \quad 0\le x_2\le 1, \quad 0\le x_3\le 1.
\end{array}
\]
Note that $X_{23}$ is entirely unconstrained in the relaxation above because its associated terms in the degree-3 BF-constraints cancel out, driven primarily by the zero bounds of $x_1$. Then, given $\alpha>0$, all solutions of the form $(\x, \X) = (0,0,0,0,0,\alpha,0)$ are feasible and allow to get arbitrarily large negative values for the objective function. However, the linearized BF-constraints for the degree-2 monomial $x_2x_3$ are given by
\[
\begin{array}{l}
-X_{23}+x_2 \ge 0\\
-X_{23}+ x_3 \ge 0\\
\phantom{-}X_{23} \ge 0\\
\phantom{-}X_{23}-x_2-x_3+1 \ge 0,
\end{array}
\]
and the first two inequalities are violated at any such point. Incorporating these inequalities into the relaxation yields an optimal value of $0$, because the system $\{0 \le x_2 \leq x_1 \leq 0, \, X_{23} \ge 0, \, X_{23} \le x_2\}$ forces $X_{23} = 0$ across all feasible solutions. Hence, the linearized BF-constraints for $x_1x_2x_3$ fail to imply those for $x_2x_3$. Further, note that $u_2-l_2>0$ and $u_3-l_3>0$, but, under the $J$-set filtering, one cannot rely on the BF-constraints associated with $x^2_2x_3$ or $x_2x^2_3$ to imply those for $x_2x_3$. The situation would be different under the $N_J^{d_J}$-set filtering, since the latter monomials would indeed be elements of $N_{\{2,3\}}^{3}$.
%LO QUE SIGUE ES INCORRECTO: Further, under the $N_J^{d_J}$-set filtering, the BF-constraints associated with $x_2x_3$ would be explicitly constructed, since $\{2,3\}\in \mathcal{M}$. NOOO!! ESTAS NO SE GENERAN!! NO tienen grado máximi de entre las que contienen a $\{2,3\}$ EN EL SENTIDO DE LA DEFINICIÓN DE ESTE FILTERING
\end{example}

The above example shows that the implications of the oversight in Lemma~2 in \cite{Sherali1992} are particularly relevant when implementing an RLT-based algorithm with some filtering of BF-constraints, as done in solvers like \solver{RLT-POS} \citep{Dalkiran2016} and \solver{RAPOSa} \citep{Gonzalez-Rodriguez:2023}. In particular, the arguments of the convergence results in \cite{Sherali2013} heavily rely on Lemma~2. The authors first present Proposition~1, which states that, for convergence of an RLT-based algorithm, it suffices to consider, for each monomial $J\in \mathcal{M}$, its associated BF-constraints. Similarly, Proposition~2 states that it also suffices to consider, for each monomial $J\in \mathcal{M}$, the bound factor constraints associated with the multisets in $N_J^{d_J}$. The proofs of these results already rely on Lemma~2 to argue that the BF-constraints associated with (sub)monomials present in those BF-constraints but not present in the original problem are also implied. Further, the authors go on to define the $J$-set routine and the $N_J^{d_J}$-set routine as two procedures that rely again on Lemma~2 to avoid generating duplicated and dominated constraints, leading to the filtering strategies described in Section~\ref{sec:filtering}.

We present below two adaptations of Lemma~2 that would allow to recover the convergence guarantees when BF-constraints are generated following the $J$-set and $N_J^{d_J}$-set routines, respectively.

\renewcommand{\thelemma}{2-$J$}

\begin{lemma}\label{lem:J-set}
Let $I$ and $J$ be two monomials with $I\subsetneq J$ and $J\in \mathcal{J}$. Assume that, for each $t \in J-I$, $l_t < u_t$.

Then, constraints $[F_{\Omega}^{\card{I}}(I_1,I_2)]_L\ge 0$, where $I_1+I_2=I$ are all implied by the constraints $[F_{\Omega}^{\card{J}}(J_1,J_2)]_L\ge 0$, where $J_1+J_2=J$.
\end{lemma}
\begin{proof}
Analogous to the proof of Lemma~2 in \cite{Sherali1992}, with the inductive argument from $I$ to $J$ applied sequentially over the variables in $J-I$.
\end{proof}

\renewcommand{\thelemma}{2-$N_J^{d_J}$}

\begin{lemma}\label{lem:NJdj-set}
Let $I$ be a monomial. Let $J$ be such that $N_I\subseteq N_{J} \subseteq \bar{J}$ for some $\bar J \in M$. Assume that $\card{I}<g(N_{J})$ and that there is $t \in N_{J}$ with $l_t < u_t$.  

Then, constraints $[F_{\Omega}^{\card{I}}(I_1,I_2)]_L\ge 0$, where $I_1+I_2=I$ are all implied by the constraints generated by the $N_J^{d_J}$-set routine.
\end{lemma}

\begin{proof}
The proof follows similar lines to the proof of Lemma~2 in \cite{Sherali1992}. We distinguish two cases.

\smallskip
\textbf{Case 1: $\boldsymbol{t\in I}$.} Since $N_I\subseteq N_{J} \subseteq \bar{J}$, we have $N_I\in \mathcal{K}$ and $\card{I}<g(N_{J})\leq g(N_I)$. Now, let $J_t=\{t, \dots, t\}$, with $\card{J_t}=g(N_I)-\card{I}\geq 1$, and let $J^*=I+J_t \in N_I^{g(N_{I})}$. Then, we can apply the inductive argument from $I$ to $J^*$ repeatedly over variable $t$.

\smallskip
\textbf{Case 2: $\boldsymbol{t\notin I}$.} Since $N_I\cup \{t\}\subseteq N_{J} \subseteq \bar{J}$, the same arguments of Case 1 hold when replacing $N_I$ with $N_I\cup \{t\}$.
\end{proof}

\renewcommand{\thelemma}{2-$N_J^{d_J}-OLD$}

%\duda{
%\begin{lemma}\label{lem:NJdj-set-OLD}
%Let $I$ be a monomial. Let $N_{J}^{g(N_{J})}$ be such that $N_I\subseteq N_{J}$ and $\card{I}<g(N_{J})$. Assume that there is $t \in N_{J}$ with $l_t < u_t$. 
%%Let $J'$ and $J$ be two monomials with $J'\subsetneq J$ and such that, for some $\bar J \in \mathcal M$, $J\in N_{\bar J}^{g(N_{\bar J})}$ and $N_{\bar J}\subseteq J$. Assume that there is $t \in J$ with $l_t < u_t$.
%
%Then, constraints $[F_{\Omega}^{\card{I}}(I_1,I_2)]_L\ge 0$, where $I_1+I_2=I$ are all implied by the constraints $[F_{\Omega}^{{g(N_{J})}}(J_1,J_2)]_L\ge 0$, where $J_1+J_2\in N_{J}^{g(N_{J})}$ and $N_J \subseteq J_1 + J_2$.
%\end{lemma}}
%\duda{
%\begin{proof}
%Analogous to the proof of Lemma~2 in \cite{Sherali1992}. It suffices to define $J_t=\{t, \dots, t\}$, with $\card{J_t}=g(N_{J})-\card{I}$ and also $J^*=I+J_t \in N_{J}^{g(N_{J})}$. Then, we can apply the inductive argument from $I$ to $J^*$ repeatedly over variable $t$.
%\end{proof}}

Recall that addressing the oversight in the original proof of Lemma~2 in \cite{Sherali1992} requires only a single variable $t \in N$ to satisfy $l_t < u_t$. Consequently, the condition in Lemma~\ref{lem:J-set} is considerably more demanding, as $l_t < u_t$ must hold for the ``extending'' variables across submonomials of maximal monomials. The assumption in Lemma~\ref{lem:NJdj-set} for the $N_J^{d_J}$-set filtering is also stronger than the one needed for Lemma~2 in \cite{Sherali1992}, but much weaker than the one needed for the $J$-set filtering.

Importantly, for any RLT-based branch-and-bound algorithm utilizing either filtering strategy, the non-degeneracy assumptions on the bounds must hold at every node of the search tree, rather than just at the root node.

\section{Practical implications}

The behavior illustrated in Example~\ref{example_orig} and Example~\ref{example_Jset} can induce inconsistencies within an RLT-based algorithm, particularly in the presence of filtering strategies such as the $J$-set and $N_J^{d_J}$-set routines. For instance, the solver may compute an incorrect lower bound at a child node that is strictly smaller than that of its parent node. Indeed, under the $J$-set strategy, in Example~\ref{example_Jset} we would get a child node with an unbounded objective function, which should never happen once the root node was bounded.

Given the above discussion, it is natural to wonder why this theoretical oversight has remained unnoticed until now, without impacting the performance of RLT-based solvers such as \solver{RLT-POS} \citep{Dalkiran2016} and \solver{RAPOSa} \citep{Gonzalez-Rodriguez:2023}. To the best of our understanding, the explanation comes from Lemma~3 in \cite{Sherali1992}, since this result implies that an RLT-based solver like the one in Figure~\ref{fig:RLTalg} will never branch on a variable that is at one of its bounds in the solution $\xb$. Thus, branching point $\beta \vcentcolon= \bar x^k_p$ will always be such that $l^k_p < \beta < u^k_p$, and so the issue in Example~\ref{example_Jset} does not arise (at least not beyond potential numerical issues when $l^k_p$ and $u^k_p$ are very close to each other).

Consequently, Lemma~3 in \cite{Sherali1992} provides a natural safeguard for continuous formulations. However, in the recent work \cite{raposainteger}, the authors present an extension of the RLT-based algorithm to mixed-integer problems, a setting where the structural vulnerability from Example~\ref{example_Jset} resurfaces. Following standard practice, when branching on a binary variable at a branching point $\beta \in (0,1)$, instead of setting bounds so that $0\leq x_j\leq \beta$ and $\beta \leq x_j \leq 1$, one should set them to $0\leq x_j\leq 0$ and $1 \leq x_j \leq 1$, obtaining tighter relaxations. Yet, as we have argued above, enforcing these discrete bounds without adjusting the underlying BF-constraints can corrupt the relaxation and induce bounding inconsistencies. Fortunately, this implementation pitfall can be easily averted. A straightforward remedy is to structurally eliminate a variable once its bounds coincide, reformulating a reduced polynomial optimization problem and recomputing its bound-factor constraints (including any filtering-specific updates, if applicable).
% A straightforward remedy is to structurally eliminate a variable once its bounds coincide, reformulating a reduced polynomial optimization problem and regenerating its corresponding BF-constraints from scratch (including the neccesary recomputations specific to filtering strategies, provided the are used). 

\section*{Acknowledgments}
We thank an anonymous referee for his contribution to improve upon an earlier version of this manuscript.

Supported through projects PID2020-116587GB-I00 funded by MINECO and the ERDF and through projects PID2021-124030NB-C32, PID2025-170122NB-C31, and PID2025-170122NB-C33 funded by MICIU/AEI/10.13039/501100011033/ and by ERDF/EU. Supported by the Xunta de Galicia (Consellería de Educación, Ciencia, Universidades e FP) under the GRC grants ED431C 2024/26 and ED431C 2025/03. Iria Rodríguez-Acevedo acknowledges the support from Consellería de Cultura, Educación, Formación Profesional e Universidades, Xunta de Galicia, through grant ED481A-2023-061. Brais González-Rodríguez acknowledges the support from MICIU, through grant BG23/00155.

\bibliographystyle{ecta}
\bibliography{references}

\end{document}